\documentclass[12pt]{amsart}

\usepackage{amssymb, amsfonts, amsmath, amscd, amsthm}
\usepackage{hyperref, eucal, enumerate, mathrsfs}
\usepackage[all]{xy}

\newtheorem{lemma}[subsection]{Lemma}
\newtheorem{proposition}[subsection]{Proposition}
\newtheorem{theorem}[subsection]{Theorem}
\newtheorem{corollary}[subsection]{Corollary}

\theoremstyle{definition}
\newtheorem{pg}[subsection]{}
\newtheorem{definition}[subsection]{Definition}
\newtheorem{remark}[subsection]{Remark}
\newtheorem{example}[subsection]{Example}

\DeclareMathOperator{\Alg}{Alg}
\DeclareMathOperator{\Br}{Br}
\DeclareMathOperator{\CAlg}{CAlg}
\DeclareMathOperator{\cg}{cg}
\DeclareMathOperator{\cn}{cn}
\DeclareMathOperator{\DD}{D}
\DeclareMathOperator{\End}{End}
\DeclareMathOperator{\Ext}{Ext}
\DeclareMathOperator{\Fun}{Fun}
\DeclareMathOperator{\Groth}{\mathsf{Groth}}
\DeclareMathOperator{\id}{id}
\DeclareMathOperator{\LinCat}{LinCat}
\DeclareMathOperator{\LMod}{LMod}
\DeclareMathOperator{\LPres}{\mathcal{P}r^{L}}
\DeclareMathOperator{\Map}{Map}
\DeclareMathOperator{\Mod}{Mod}
\DeclareMathOperator{\Pres}{\mathcal{P}r}
\DeclareMathOperator{\PSt}{PSt}
\DeclareMathOperator{\QCoh}{QCoh}
\DeclareMathOperator{\QStk}{QStk}
\DeclareMathOperator{\rev}{rev}
\DeclareMathOperator{\RMod}{RMod}
\DeclareMathOperator{\sBr}{\mathscr{B}r}
\DeclareMathOperator{\Sp}{Sp}
\DeclareMathOperator{\Spec}{Spec}
\DeclareMathOperator{\SSet}{\mathcal{S}}
\newcommand{\Sphere}{S}
\DeclareMathOperator{\St}{St}
\DeclareMathOperator{\Z}{\mathfrak{Z}}

\DeclareMathOperator{\calA}{\mathcal{A}}
\DeclareMathOperator{\calB}{\mathcal{B}}
\DeclareMathOperator{\calC}{\mathcal{C}}
\DeclareMathOperator{\calD}{\mathcal{D}}
\DeclareMathOperator{\calF}{\mathcal{F}}
\DeclareMathOperator{\calO}{\mathcal{O}}
\DeclareMathOperator{\calQ}{\mathcal{Q}}
\DeclareMathOperator{\calX}{\mathcal{X}}

\DeclareMathOperator{\bbE}{\mathbb{E}}
\DeclareMathOperator{\bbF}{\mathbb{F}}
\DeclareMathOperator{\bbZ}{\mathbb{Z}}

\DeclareMathOperator{\sfX}{\mathsf{X}}
\DeclareMathOperator{\sfY}{\mathsf{Y}}

\newcommand{\ldual}[1]{{}^\vee{#1}}
\newcommand{\BMod}[2]{{}_{#1}\text{BMod}_{#2}}

\newcommand{\Adjoint}[4]{\xymatrix@1{#1:#2 \ar@<.4ex>[r] & #3:#4 \ar@<.4ex>[l]}}
\newcommand{\Pull}[8]{\xymatrix{#1\ar[r]^-{#5}\ar[d]^-{#6} & #2 \ar[d]^-{#7} \\ #3 \ar[r]^-{#8} & #4}}

\title{Twisted Equivalences in Spectral Algebraic Geometry}
\author{Chang-Yeon Chough}
\address{Center for Quantum Structures in Modules and Spaces, Seoul National University, Seoul 08826, Republic of Korea}
\email{chough@snu.ac.kr}

\begin{document}

\begin{abstract}
	We study twisted derived equivalences for schemes in the setting of spectral algebraic geometry. To this end, we introduce the notion of a twisted equivalence and show that a twisted equivalence for perfect spectral algebraic stacks admitting a quasi-finite presentation supplies an equivalence between the stacks, which compensate for the failure of twisted derived equivalences for non-affine schemes to provide an isomorphism of the schemes. In the case of (not necessarily connective) commutative ring spectra, we also prove a spectral analogue of Rickard's theorem, which shows that a derived equivalence of associative rings induces an isomorphism between their centers.
\end{abstract}

\setcounter{tocdepth}{1} 
\maketitle
\tableofcontents

\section{Introduction}

\begin{pg}
	Let $A$ and $B$ be commutative rings and let $\alpha$ and $\beta$ be elements of the Brauer groups of $A$ and $B$, respectively (see \cite[1.2]{MR1608798}). Suppose we are given a triangulated equivalence between the derived category of $\alpha$-twisted $A$-modules and the derived category of $\beta$-twisted $B$-modules; see \cite[1.2.1]{MR2700538}. Then there exists an isomorphism of rings $f: A \rightarrow B$ such that $f^\ast(\alpha)=\beta$ in the Brauer group of $B$, generalizing \cite[p.177]{MR1887894} of Andrei C\u{a}ld\u{a}raru (see, for example, \cite[3.5]{MR3616004}). However, Benjamin Antieau's example of a K3 surface shows that this result does not generalize to non-affine schemes; see \cite[p.11]{MR3616004}. 

	The main goal of this paper is to establish the following analogue of \cite[3.5]{MR3616004} in the setting of spectral algebraic geometry, which is valid for a large class of algebro-geometric objects, rather than merely for affine spectral Deligne-Mumford stacks of \cite[\textsection 1.4.7]{lurie2018sag} (note that we will see in \ref{twisted equivalences for Azumaya algebras in the case of nonconnective affines} that a similar statement also holds for affine nonconnective spectral Deligne-Mumford stacks):
\end{pg}

\begin{theorem}\label{twisted equivalences induce equivalences}
	Let $f:X \rightarrow Y$ be a morphism of perfect spectral algebraic stacks which admit a quasi-finite presentation. Let $\calC$ and $\calD$ be objects of the extended Brauer spaces of $X$ and $Y$, respectively. Suppose we are given a twisted equivalence $F:\QCoh(Y; \calD) \rightarrow \QCoh(X; \calC)$ between the $\infty$-categories of global sections of quasi-coherent stacks $\calC$ and $\calD$, in the sense of \emph{\ref{twisted equivalences}}. Then $f$ is an equivalence, and there exists an equivalence $f^\ast \calD \simeq \calC$ of quasi-coherent stacks on $X$. 
\end{theorem}

\begin{pg} 
	Using $\bbE_\infty$-rings (in the sense of \cite[7.1.0.1]{lurie2017ha}) in place of ordinary commutative rings, Jacob Lurie set up the foundation of spectral algebraic geometry, which provides a natural framework to understand the derived $\infty$-categories of algebro-geometric objects; see \cite{lurie2018sag}. In the spectral setting of \ref{twisted equivalences induce equivalences}, we use the extended Brauer spaces of \cite[11.5.2.1]{lurie2018sag} and the $\infty$-categories of global sections of quasi-coherent stacks of \cite[\textsection 10.4.1]{lurie2018sag} (see also \cite[5.9]{chough_brauer}) in place of the Brauer groups and the derived categories of twisted modules, respectively.
\end{pg}

\begin{pg}
	The notion of a \emph{perfect spectral algebraic stack} of \ref{perfect spectral algebraic stacks} is a spectral analogue of the notion of an ordinary quasi-compact algebraic stack $\calX$ for which the diagonal map $\calX \rightarrow \calX \times \calX$ is quasi-affine, the derived category $\DD_{\mathrm{qc}}(\calX)$ of quasi-coherent sheaves on $\calX$ is compactly generated, and the structure sheaf $\calO_{\calX} \in \DD_{\mathrm{qc}}(\calX)$ is compact. If $X$ is a perfect spectral algebraic stack, we say that it admits a \emph{quasi-finite presentation} if there exists a morphism $\Spec A \rightarrow X$ which is locally quasi-finite, faithfully flat, and locally almost of finite presentation, where $A$ is a connective $\bbE_\infty$-ring (see \cite[1.3]{chough_brauer}). By virtue of \ref{examples of perfect spectral algebraic stacks}, the class of perfect spectral algebraic stacks admitting a quasi-finite presentation includes all quasi-compact spectral Deligne-Mumford stacks $\sfX$ for which the diagonal map $\sfX \rightarrow \sfX \times \sfX$ is quasi-affine and the structure sheaf $\calO_{\sfX}$ is a compact object of the $\infty$-category $\QCoh(\sfX)$ of quasi-coherent sheaves on $\sfX$ of \cite[2.2.2.1]{lurie2018sag} (see \cite[1.4.4.2]{lurie2018sag} for the definition of a spectral Deligne-Mumford stack). In particular, the class includes every quasi-compact and quasi-separated spectral algebraic space (see \cite[1.6.8.1]{lurie2018sag}). Moreover, if $R$ is an ordinary commutative ring, then the class includes the underlying quasi-geometric spectral algebraic stacks of perfect derived algebraic stacks over $R$ which admit a quasi-finite presentation (see \ref{perfect derived algebraic stacks}).
\end{pg}

\begin{remark}
	Using the linear equivalence between the ordinary categories of twisted quasi-coherent sheaves in place of the triangulated equivalence between the derived categories of twisted modules in the statement of \cite[3.5]{MR3616004}, Antieau's generalization of C\u{a}ld\u{a}raru's conjecture of \cite[p.173]{MR1887894} shows that an analogous statement holds for quasi-compact and quasi-separated schemes; see \cite[1.1]{MR3466552}. In the case where the twists are given by $\mathbb{G}_m$-gerbes, John Calabrese and Michael Groechenig proved an analogue for quasi-compact and separated algebraic spaces; see \cite[3.5]{MR3322195}. We can also regard \ref{twisted equivalences induce equivalences} as an analogue of these results in the setting of spectral algebraic geometry.
\end{remark}

\begin{example}
	Let $f:X \rightarrow Y$ be as in \ref{twisted equivalences induce equivalences}. Suppose that $F$ is representable in the sense of \cite[6.3.2.1]{lurie2018sag} and that there exists an invertible compactly generated stable quasi-coherent stack $\calC$ on $X$ (that is, $\calC$ is an object of the extended Brauer space of $X$) for which the pushforward $f_\ast \calC$ (see \cite[10.1.4.1]{lurie2018sag}) is an object the extended Brauer space of $Y$. Then the canonical equivalence $\QCoh(Y; f_\ast \calC) \rightarrow \QCoh(X; \calC)$ is a twisted equivalence (see \cite[10.1.7.4]{lurie2018sag}), so that the morphism $f$ must be an equivalence by virtue of \ref{twisted equivalences induce equivalences}. In particular, if $f: \sfX \rightarrow \sfY$ is a map of quasi-compact and quasi-separated spectral algebraic spaces for which there exists an Azumaya algebra on $\sfX$ whose pushforward along $f$ is an Azumaya algebra on $\sfY$ (see \cite[11.5.3.7]{lurie2018sag} and \cite[2.5.4.3]{lurie2018sag}), then $f$ is an equivalence.
\end{example}

\begin{remark}
	In contrast with \cite[3.5]{MR3616004} and \cite[1.1]{MR3466552} (see also \cite[3.5]{MR3322195}) in classical algebraic geometry, we impose some additional conditions in the spectral setting of \ref{twisted equivalences induce equivalences}: 
\begin{enumerate}[(i)]
\item\label{To give a morphism} We assume \emph{a priori} that a morphism $f: X \rightarrow Y$ is given.

\item\label{To equip the global sections with the actions} In the definition of a twisted equivalence of \ref{twisted equivalences}, we regard the $\infty$-category of global sections $\QCoh(X; \calC)$ not just as a presentable stable $\infty$-category (see \cite[5.5.0.1]{MR2522659} and \cite[1.1.1.9]{lurie2017ha}), but also as equipped with an action of $\QCoh(X)$ (and similarly for $\QCoh(Y; \calD)$). 

\item\label{twisted equivalences are stronger than equivalences} The condition that $F: \QCoh(Y; \calD) \rightarrow \QCoh(X; \calC)$ is a twisted equivalence is stronger than the condition that $F$ is an equivalence of stable $\infty$-categories (or even that $F$ is an equivalence of $\QCoh(Y)$-module objects of the symmetric monoidal $\infty$-category $\Pres^{\St}$ of presentable stable $\infty$-categories of \cite[4.8.2.18]{lurie2017ha}, where we regard $\QCoh(X; \calC)$ as a module over $\QCoh(Y)$ via the pullback functor $f^\ast: \QCoh(Y) \rightarrow \QCoh(X)$ of \cite[6.2.2.6]{lurie2018sag}).
\end{enumerate}
Let us give an account of the necessity of each of these conditions:
\end{remark}

\begin{pg}\label{an explanation for the actions}
	To understand condition (\ref{To equip the global sections with the actions}), let us first consider the special case of \ref{twisted equivalences induce equivalences} in which $f$ is an identity morphism on $X$. In this case, the necessity of equipping $\QCoh(X; \calC)$ and $\QCoh(X; \calD)$ with the actions of $\QCoh(X)$ follows from \ref{fully faithfulness of the global sections functor}, which asserts that if $X$ is a quasi-geometric stack satisfying condition $(\ast)$ of \ref{promotion of the global sections functors}, then a morphism $\calC \rightarrow \calD$ of stable quasi-coherent stacks on $X$ is an equivalence if and only if the induced map $\QCoh(X; \calC) \rightarrow \QCoh(X; \calD)$ is an equivalence of $\QCoh(X)$-module objects of $\Pres^{\St}$. 
\end{pg}

\begin{pg}\label{another explanation for the actions}
	As another illustration of the perspective that it is necessary to equip $\QCoh(X; \calC)$ with the action of $\QCoh(X)$, we consider the case where $X=\Spec B$ and $Y=\Spec A$ are affine spectral Deligne-Mumford stacks. Take $\calC$ and $\calD$ to be the unit objects of the $\infty$-categories $\QStk^{\St}(X)$ and $\QStk^{\St}(Y)$ of stable quasi-coherent stacks on $X$ and $Y$, respectively (see \cite[10.1.2.1]{lurie2018sag} and \cite[10.1.6.4]{lurie2018sag}). Then we can identify $\QCoh(X; \calC) \in \Mod_{\QCoh(X)}(\Pres^{\St})$ with the stable $\infty$-category $\Mod_B$ of $B$-module spectra of \cite[7.1.1.2]{lurie2017ha} equipped with the action of $\Mod_B$ on itself via the symmetric monoidal structure on $\Mod_B$ of \cite[4.5.2.1]{lurie2017ha} (and similarly for $\QCoh(Y; \calD)$). Now the condition that $F: \Mod_A \rightarrow \Mod_B$ is an $A$-linear equivalence (rather than merely an equivalence of $\infty$-categories) in the definition of a twisted equivalence of \ref{twisted equivalences} is equivalent to the condition that $F$ is a symmetric monoidal equivalence between the symmetric monoidal $\infty$-categories $\Mod_A$ and $\Mod_B$ (note that the commutativity of the diagram appearing in \ref{twisted equivalences} allows us to identify $F$ with the extension of scalars functor $\Mod_A \rightarrow \Mod_B$ associated to the morphism $f$). From this point of view, it is necessary to impose condition (\ref{To equip the global sections with the actions}), because $f$ is an equivalence if and only if $F$ is a symmetric monoidal equivalence (see \cite[7.1.2.7]{lurie2017ha}). 
\end{pg}

\begin{remark}
	In view of \ref{an explanation for the actions} and \ref{another explanation for the actions}, we need to endow the $\QCoh(X)$-module $\QCoh(X; \calC)$ in $\Pres^{\St}$ the structure of a module over $\QCoh(Y)$, so that $F: \QCoh(Y; \calD) \rightarrow \QCoh(X; \calC)$ appearing in \ref{twisted equivalences} can be viewed as a $\QCoh(Y)$-linear functor. For this, we need a morphism $\Phi: \QCoh(Y) \rightarrow \QCoh(X)$ of commutative algebra objects in the symmetric monoidal $\infty$-category $\Pres^{\St}$. In the special case where $X$ is a quasi-compact and quasi-separated spectral algebraic space, the Tannaka duality for spectral algebraic spaces of \cite[9.6.0.1]{lurie2018sag} guarantees that giving a morphism $\Phi$ is equivalent to giving a morphism of functors $X \rightarrow Y$, thereby providing justification for condition (\ref{To give a morphism}).
\end{remark}

\begin{pg}
	Before turning to condition (\ref{twisted equivalences are stronger than equivalences}), we establish a close connection between the definition of a twisted equivalence of \ref{twisted equivalences} and the center of $\bbE_1$-algebras in the sense of \cite[5.3.1.12]{lurie2017ha}. Let $k$ be a nonnegative integer. If $R$ is an $\bbE_\infty$-ring, then we say that $A$ is an \emph{$\bbE_k$-algebra over $R$} if it is an $\bbE_k$-algebra object of the $\infty$-category $\Mod_R$ (see \cite[7.1.3.5]{lurie2017ha}). The following assertion is a spectral analogue of \cite[9.2]{MR1002456} (which shows that if two ordinary associative rings are derived equivalent in the sense of \cite[6.4]{MR1002456}, then there is an isomorphism between the centers of the rings):
\end{pg}

\begin{theorem}\label{twisted equivalences in the case of nonconnective affines}
	Let $f:R' \rightarrow R$ be a morphism of $\bbE_\infty$-rings, and let $A'$ and $A$ be $\bbE_1$-algebras over $R'$ and $R$, respectively. Suppose we are given a twisted equivalence $F: \RMod_{A'} \rightarrow \RMod_A$ between the $\infty$-categories of right module spectra over $A'$ and $A$, in the sense of \emph{\ref{setup for twisted equivalences for nonconnective affines}}. Then the induced map $\Z(A') \rightarrow \Z(A)$ is an equivalence of $\bbE_2$-algebras over $R'$, where we regard $\Z(A)$ as an $\bbE_2$-algebra over $R'$ via the map $f$.
\end{theorem}

\begin{remark}
	Here $\Z(A)$ denotes the center of the $\bbE_1$-algebra $A$ over $R$ in the sense of \cite[5.3.1.12]{lurie2017ha}, and $\Z(A')$ is defined similarly. By virtue of \cite[5.3.1.30]{lurie2017ha} and \cite[4.4.1.28]{lurie2017ha}, the underlying $\bbE_1$-algebra over $R$ of the $\bbE_2$-algebra $\Z(A)$ over $R$ classifies endomorphisms of $A$ as an object of the $\infty$-category $\BMod{A}{A}(\Mod_R)$ of $A$-$A$-bimodule objects of $\Mod_R$ (regarded as an $\infty$-category tensored over $\Mod_R$); see \cite[4.3.1.12]{lurie2017ha}. In particular, if $R$ is the sphere spectrum, then the homotopy groups $\pi_n \Z(A)$ can be identified with the \emph{$n$th topological Hochschild cohomology groups} of the $\bbE_1$-ring $A$. 
\end{remark}

\begin{remark}
	Unlike \ref{twisted equivalences} in the affine case, we allow affine nonconnective spectral Deligne-Mumford stacks in \ref{setup for twisted equivalences for nonconnective affines}. In other words, we do not require the $\bbE_\infty$-rings appearing in \ref{setup for twisted equivalences for nonconnective affines} to be connective (see \cite[p.1201]{lurie2017ha}). In fact, \ref{setup for twisted equivalences for nonconnective affines} can be regarded as a special case of \ref{twisted equivalences}, provided that we extend the theory of quasi-coherent stacks to all functors from the $\infty$-category $\CAlg$ of $\bbE_\infty$-rings (see \cite[7.1.0.1]{lurie2017ha}) to the $\infty$-category $\widehat{\SSet}$ of spaces which are not necessarily small (see \cite[1.2.16.4]{MR2522659}). 
\end{remark}

\begin{pg}\label{the centers and the forgetful functors}
	We now turn to the issue involved in condition (\ref{twisted equivalences are stronger than equivalences}). We explain this issue in the context of \ref{twisted equivalences in the case of nonconnective affines}: let $R$ be a connective $\bbE_\infty$-ring and consider the unit map $\Sphere \rightarrow R$, where $\Sphere$ is the sphere spectrum. Let $\underline{R}$ denote the underlying $\bbE_1$-ring of $R$, so that the canonical map $\theta: \RMod_{\underline{R}} \rightarrow \RMod_R$, given on objects by the identity, is an equivalence of stable $\infty$-categories. However, the induced map $\Z(\underline{R}) \rightarrow \underline{\Z(R)}$ is generally not an equivalence of $\bbE_2$-rings, where $\underline{\Z(R)}$ denotes the underlying $\bbE_2$-ring of $\Z(R)$. To see this, we observe that we can identify $\Z(R)$ with the underlying $\bbE_2$-ring of $R$ (since $R$ is an $\bbE_\infty$-ring), and that the homotopy groups $\pi_\ast \Z(\underline{R})$ can be identified with the topological Hochschild cohomology groups $\Ext^{-\ast}_{\BMod{\underline{R}}{\underline{R}}(\Sp)}(\underline{R}, \underline{R})$ of $\underline{R}$, where $\BMod{\underline{R}}{\underline{R}}(\Sp)$ denotes the $\infty$-category of $\underline{R}$-$\underline{R}$-bimodule spectra (see \cite[4.3.1.12]{lurie2017ha}). Now let $p$ be a prime number and take $R$ to be the finite field $\bbF_p$ with $p$ element, regarded as a discrete $\bbE_\infty$-ring. In this case, the map $\Z(\underline{R}) \rightarrow \underline{\Z(R)}$ is not an equivalence, because $\pi_2 \Z(R) \simeq 0$ and $\pi_2 \Z(\underline{R}) \simeq \bbF_p$ (see \cite[1.1]{bokstedt1985thh}). We conclude that the condition that the functor $F$ appearing in \ref{twisted equivalences in the case of nonconnective affines} is an equivalence of stable $R'$-linear $\infty$-categories alone is not sufficient to deduce \ref{twisted equivalences in the case of nonconnective affines}: we will need the full strength of our assumption that $F$ is a twisted equivalence of \ref{setup for twisted equivalences for nonconnective affines}.
\end{pg}
 
\begin{pg}\textbf{Conventions}.
	We will make use of the language of $\infty$-categories and the theory of spectral algebraic geometry developed in \cite{MR2522659}, \cite{lurie2017ha}, and \cite{lurie2018sag}. We will adopt the set-theoretic convention of \cite{MR2522659}.
\end{pg}

\begin{pg}\textbf{Acknowledgements}. 
	The author is thankful to Benjamin Antieau for helpful comments. This work was supported by the National Research Foundation of Korea (NRF) grant funded by the Korean government (MSIT) (No. 2020R1A5A1016126).
\end{pg}

\section{Twisted Equivalences for Spectral Algebraic Stacks}

	Our goal in this section is to establish our main result \ref{twisted equivalences induce equivalences}. 

\begin{pg}
	We begin by introducing our principal objects of interest to which \ref{twisted equivalences induce equivalences} applies. Let $0 \leq k \leq \infty$. According to \cite[7.1.0.1]{lurie2017ha}, an \emph{$\bbE_k$-ring} is an $\bbE_k$-algebra object of the symmetric monoidal $\infty$-category $\Sp$ of spectra (see \cite[4.8.2.19]{lurie2017ha}). We will be primarily interested in the cases $k=1,2$, and $\infty$. Recall that an $\bbE_\infty$-ring $R$ is said to be \emph{connective} if its underlying spectrum is connective; see \cite[p.1201]{lurie2017ha}. Let $R$ be a connective $\bbE_\infty$-ring. We let $\CAlg^{\cn}_R$ denote the $\infty$-category of connective $\bbE_\infty$-algebras over $R$; see \cite[7.1.3.8]{lurie2017ha}. In the special case where $R$ is the sphere spectrum, we will denote $\CAlg^{\cn}_R$ simply by $\CAlg^{\cn}$. Recall from \cite[2.6]{chough_brauer} that a \emph{quasi-geometric spectral algebraic stack over $R$} is a functor $X: \CAlg^{\cn}_R \rightarrow \widehat{\SSet}$ satisfying the following conditions (here $\widehat{\SSet}$ denotes the $\infty$-category of (not necessarily small) spaces of \cite[1.2.16.4]{MR2522659}):
\begin{enumerate}[(i)]
\item The functor $X$ satisfies descent for the fpqc topology of \cite[B.6.1.3]{lurie2018sag}.
\item The diagonal map $\Delta: X \rightarrow X \times X$ is representable and quasi-affine in the sense of \cite[6.3.3.6]{lurie2018sag}.
\item There exist a connective $\bbE_\infty$-algebra $A$ over $R$ and a fiber smooth and surjective morphism $\Spec A \rightarrow X$; see \cite[11.2.5.5]{lurie2018sag}.
\end{enumerate}
\end{pg}

\begin{pg}
	To each functor $X: \CAlg^{\cn} \rightarrow \widehat{\SSet}$, we can associate an $\infty$-category $\QCoh(X)$, which we call the \emph{$\infty$-category of quasi-coherent sheaves on $X$}. More informally, an object $\calF \in \QCoh(X)$ is a rule which assigns to each point $\eta \in X(A)$ an $A$-module $\calF_\eta$, depending functorially on the pair $(A, \eta)$; see \cite[6.2.2.1]{lurie2018sag}. In a relative setting, we define an $\infty$-category of quasi-coherent sheaves as follows:
\end{pg}

\begin{definition}\label{qcoh in the relative setting}
	Let $R$ be a connective $\bbE_\infty$-ring. Let $X: \CAlg^{\cn}_R \rightarrow \widehat{\SSet}$ be a functor, and let $\overline{X}$ denote the image of $X$ under the equivalence of $\infty$-categories $\Fun(\CAlg^{\cn}_R, \widehat{\SSet}) \simeq \Fun(\CAlg^{\cn}, \widehat{\SSet})_{/\Spec R}$, where the latter $\infty$-category denotes the slice $\infty$-category. We define the \emph{$\infty$-category $\QCoh(X)$ of quasi-coherent sheaves on $X$} to be the $\infty$-category $\QCoh(\overline{X})$ of quasi-coherent sheaves on $\overline{X}$ in the sense of \cite[6.2.2.1]{lurie2018sag}. 
\end{definition}

\begin{pg}
	We now introduce the class of algebro-geometric objects that we are interested in:
\end{pg}

\begin{definition}\label{perfect spectral algebraic stacks}
	Let $R$ be a connective $\bbE_\infty$-ring. We will say that a functor $X: \CAlg^{\cn}_R \rightarrow \widehat{\SSet}$ is a \emph{perfect spectral algebraic stack over $R$} if it is a quasi-geometric spectral algebraic stack over $R$ (in the sense of \cite[2.6]{chough_brauer}) which satisfies the following conditions:
\begin{enumerate}[(i)]
\item The $\infty$-category $\QCoh(X)$ (in the sense of \ref{qcoh in the relative setting}) is compactly generated (see \cite[5.5.7.1]{MR2522659}).
\item The structure sheaf $\calO_X \in \QCoh(X)$ is compact (see \cite[5.3.4.5]{MR2522659}).
\end{enumerate}
\end{definition}

\begin{remark}\label{perfect spectral algebraic stacks over the sphere spectrum} \noindent 
\begin{enumerate}[(i)]
\item Using the notation of \ref{qcoh in the relative setting}, we see that a functor $X: \CAlg^{\cn}_R \rightarrow \widehat{\SSet}$ is a perfect spectral algebraic stack over $R$ if and only if $\overline{X}$ is a perfect spectral algebraic stack over the sphere spectrum $S$; see \cite[2.7]{chough_brauer}.

\item In the special case where $R$ is the sphere spectrum, $X$ is a perfect spectral algebraic stack (over the sphere spectrum) if and only if $X$ is both a quasi-geometric spectral algebraic stack (over the sphere spectrum) and a perfect stack in the sense of \cite[9.4.4.1]{lurie2018sag} by virtue of \cite[9.4.4.5]{lurie2018sag}. 
\end{enumerate}
\end{remark}

\begin{pg}
	There is an entirely parallel story in the setting of derived algebraic geometry: let $R$ be a discrete $\bbE_\infty$-ring (that is, an ordinary commutative ring) and let $\CAlg^\Delta_R$ denote the $\infty$-category of simplicial commutative $R$-algebras (see \cite[25.1.1.1]{lurie2018sag}). According to \cite[2.25]{chough_brauer}, a \emph{quasi-geometric derived algebraic stack over $R$} is a functor $X: \CAlg^\Delta_R \rightarrow \widehat{\SSet}$ which satisfies the following conditions:
\begin{enumerate}[(i)]
\item The functor $X$ satisfies descent for the fpqc topology of \cite[2.17]{chough_brauer}.
\item The diagonal map $\Delta: X \rightarrow X \times X$ is representable and quasi-affine.
\item There exist a simplicial commutative $R$-algebra $A$ and a smooth and surjective morphism $\Spec A \rightarrow X$.
\end{enumerate}
\end{pg}

\begin{pg}
	To each quasi-geometric derived algebraic stack $X$ over $R$, we can associate a quasi-geometric spectral algebraic stack $X^\circ$ over $R$, which we refer to as the \emph{underlying quasi-geometric spectral algebraic stack of $X$ over $R$}; see \cite[2.31]{chough_brauer}. Note that if $X=\Spec A$ is affine, then $X^\circ$ can be identified with $\Spec A^\circ$, where $A^\circ$ denotes the image of $A$ under the forgetful functor $\Theta_R: \CAlg^\Delta_R \rightarrow \CAlg^{\cn}_R$ of \cite[25.1.2.1]{lurie2018sag}.
\end{pg}

\begin{definition}\label{perfect derived algebraic stacks}
	Let $R$ be an ordinary commutative ring. We will say that a functor $X: \CAlg^\Delta_R \rightarrow \widehat{\SSet}$ is a \emph{perfect derived algebraic stack over $R$} if it is a quasi-geometric derived algebraic stack over $R$ (in the sense of \cite[2.25]{chough_brauer}) and the underlying quasi-geometric spectral algebraic stack $X^\circ$ over $R$ is perfect in the sense of \ref{perfect spectral algebraic stacks}.
\end{definition}

\begin{remark}
	The property of being a perfect derived algebraic stack can also be described as in \ref{perfect spectral algebraic stacks} by virtue of \cite[5.4]{chough_brauer}.
\end{remark}

\begin{example}\label{examples of perfect spectral algebraic stacks}
	The class of perfect spectral algebraic stacks includes many algebro-geometric objects of interest:
\begin{enumerate}[(i)]
\item Let $\sfX$ be a quasi-geometric spectral Deligne-Mumford stack of \cite[9.1.4.1]{lurie2018sag} (that is, $\sfX$ is a quasi-compact spectral Deligne-Mumford stack for which the diagonal map $\sfX \rightarrow \sfX \times \sfX$ is quasi-affine), so that $\QCoh(\sfX)$ is compactly generated by virtue of \cite[1.10]{chough_brauer}. If the structure sheaf $\calO_{\sfX} \in \QCoh(\sfX)$ is compact, then $\sfX$ can be regarded as a perfect spectral algebraic stack. In particular, it follows from \cite[9.6.1.1]{lurie2018sag} that every quasi-compact and quasi-separated spectral algebraic space (see \cite[1.6.8.1]{lurie2018sag}) is a perfect spectral algebraic stack. 

\item Let $R$ be an ordinary commutative ring. Then for each perfect derived algebraic stack over $R$, its underlying quasi-geometric spectral algebraic stack is a perfect spectral algebraic stack over $R$. In particular, if $\sfX$ is a quasi-compact derived Deligne-Mumford stack with quasi-affine diagonal, whose structure sheaf $\calO_{\sfX}$ is a compact object of $\QCoh(\sfX)$, then its underlying spectral Deligne-Mumford stack can be viewed as a perfect spectral algebraic stack (see \cite[2.18]{chough_brauer} for the definitions of a derived Deligne-Mumford stack and its underlying spectral algebraic stack).
\end{enumerate}
\end{example}

\begin{pg}
	We now give a quick review of some terminology which will be needed to define the notion of a twisted equivalence of \ref{twisted equivalences}. Let $\LPres$ denote the $\infty$-category of presentable $\infty$-categories (see \cite[5.5.3.1]{MR2522659}), which we regard as endowed with the symmetric monoidal structure of \cite[4.8.1.15]{lurie2017ha}. For each $\bbE_\infty$-ring $A$, we let $\Mod_A$ denote the $\infty$-category of $A$-module spectra of \cite[7.1.1.2]{lurie2017ha}. We regard $\Mod_A$ as equipped with the symmetric monoidal structure of \cite[7.1.3.8]{lurie2017ha}, so that it can be viewed as a commutative algebra object of the symmetric monoidal $\infty$-category $\LPres$. Let $\LinCat^{\St}_A$ denote the $\infty$-category $\Mod_{\Mod_A}(\LPres)$, which we refer to as the \emph{$\infty$-category of stable $A$-linear $\infty$-categories}; see \cite[D.1.5.1]{lurie2018sag}. If $X: \CAlg^{\cn} \rightarrow \widehat{\SSet}$ is a functor, we can associate to $X$ an \emph{$\infty$-category $\QStk^{\St}(X)$ of stable quasi-coherent stacks on $X$}. Roughly speaking, an object $\calC \in \QStk^{\St}(X)$ is a rule which assigns to each point $\eta \in X(A)$ a stable $A$-linear $\infty$-category $\calC_\eta$, which depends functorially on the pair $(A, \eta)$; see \cite[10.1.2.1]{lurie2018sag}. Note that we can regard $\QStk^{\St}(X)$ as a spectral analogue of the twisted derived category of a scheme, where the twist is given by an element of the $2$nd \'etale cohomology group of the scheme with coefficient $\mathbb{G}_m$.
\end{pg}

\begin{pg}\label{promotion of the global sections functors}
	Let us recall the definition of a global section of a quasi-coherent stack (see \cite[5.9]{chough_brauer}). For this, let $X$ be a quasi-geometric stack which satisfies the following condition:
\begin{itemize}
\item[$(\ast)$] 
There exist an affine spectral Deligne-Mumford stack $\sfX_0$ and a morphism of quasi-geometric stacks $\sfX_0 \rightarrow X$ which is faithfully flat and locally almost of finite presentation (see \cite[4.2.0.1]{lurie2018sag}).
\end{itemize}
Let $q: X \rightarrow \Spec \Sphere$ denote the projection map, where $\Sphere$ is the sphere spectrum. Let $\QStk^{\PSt}(X)$ denote the $\infty$-category of prestable quasi-coherent stacks on $X$; see \cite[10.1.2.1]{lurie2018sag}. In the special case where $X=\Spec S$, we can identify $\QStk^{\PSt}(X)$ with the $\infty$-category $\Groth_\infty$ of Grothendieck prestable $\infty$-categories of \cite[C.3.0.5]{lurie2018sag}. Using \cite[10.4.1.1]{lurie2018sag}, we see that the pullback functor $q^\ast: \Groth_\infty \rightarrow \QStk^{\PSt}(X)$ admits a right adjoint $\QCoh(X; \bullet): \QStk^{\PSt}(X) \rightarrow \Groth_\infty$ which we refer to as the \emph{global sections functor on $X$}. By virtue of \cite[5.11]{chough_brauer}, the adjunction between $q^\ast$ and $\QCoh(X; \bullet)$ restricts to a pair of adjoint functors
$$
\Adjoint{q^\ast}{\Pres^{\St}}{\QStk^{\St}(X)}{\QCoh(X; \bullet)},
$$
where $\Pres^{\St}$ denotes the $\infty$-category of presentable stable $\infty$-categories of \cite[4.8.2.18]{lurie2017ha}. Let us regard $\Pres^{\St}$ and $\QStk^{\St}(X)$ as equipped with the symmetric monoidal structures of \cite[4.8.2.18]{lurie2017ha} and \cite[10.1.6.4]{lurie2018sag}, respectively. Then $q^\ast: \Pres^{\St} \rightarrow \QStk^{\St}(X)$ is symmetric monoidal, so that the functor $\QCoh(X; \bullet): \QStk^{\St}(X) \rightarrow \Pres^{\St}$ is lax symmetric monoidal. We can therefore promote the global sections functor $\QCoh(X; \bullet)$ to a functor
$$
\QStk^{\St}(X) \simeq \Mod_{\calQ_X}(\QStk^{\St}(X)) \stackrel{\QCoh(X; \bullet)}{\longrightarrow} \Mod_{\QCoh(X; \calQ_X)}(\Pres^{\St}) \simeq \Mod_{\QCoh(X)}(\Pres^{\St}),
$$
where $\calQ_X$ denotes the unit object of the symmetric monoidal $\infty$-category $\QStk^{\St}(X)$. We will abuse notation by denoting this functor also by $\QCoh(X; \bullet)$. 
\end{pg}

\begin{pg}
	In what follows, we will use the theory of duality in the setting of monoidal $\infty$-categories developed in \cite[\textsection 4.6.1]{lurie2017ha}. We first recall a bit of terminology which will be needed later in \ref{pushforward functors and Frobenius algebra objects}. According to \cite[D.7.4.1]{lurie2018sag}, a monoidal stable $\infty$-category $\calC$ is said to be \emph{locally rigid} if it satisfies the following conditions:  
\begin{enumerate}[(i)]
\item The $\infty$-category $\calC$ is compactly generated; see \cite[5.5.7.1]{MR2522659}.
\item The tensor product on $\calC$ preserves small colimits separately in each variable.
\item The unit object $\mathbf{1}$ of $\calC$ is compact; see \cite[5.3.4.5]{MR2522659}.
\item Every compact object of $\calC$ admits both a left and a right dual (see \cite[4.6.1.1]{lurie2017ha}). 
\end{enumerate}
Now let $\calC$ be a monoidal $\infty$-category which admits geometric realizations of simplicial objects and the tensor product functor $\otimes: \calC \times \calC \rightarrow \calC$ preserves geometric realizations of simplicial objects separately in each variable, let $A$ be an associative algebra object of $\calC$, and let $\lambda: A \rightarrow \mathbf{1}$ be a morphism in $\calC$ (here $\mathbf{1}$ denotes the unit object of $\calC$). We will say that a pair $(A, \lambda)$ is a \emph{Frobenius algebra object} of $\calC$ if the composition of $\lambda$ with the multiplication map $A \otimes A \rightarrow A$ is a duality datum in $\calC$; see \cite[4.6.5.1]{lurie2017ha}.
\end{pg}

\begin{pg}
	Let $f: X \rightarrow Y$ be a morphism of quasi-geometric stacks satisfying condition $(\ast)$ of \ref{promotion of the global sections functors}. Suppose we are given a stable quasi-coherent stack $\calC$ on $X$ for which the $\infty$-category of global sections $\QCoh(X; \calC)$ is dualizable as an object of $\Mod_{\QCoh(X)}(\Pres^{\St})$; see \ref{promotion of the global sections functors}. Let $\ldual{\QCoh(X; \calC)}$ denote the dual of $\QCoh(X; \calC)$. Before introducing the notion of a twisted equivalence of \ref{twisted equivalences}, we will need to understand $\ldual{\QCoh(X; \calC)}$ as a module over $\QCoh(Y)$ (via the pullback functor $f^\ast: \QCoh(Y) \rightarrow \QCoh(X)$ of \cite[6.2.2.6]{lurie2018sag}). For this, we establish the following relative version of \cite[D.7.5.1]{lurie2018sag}:
\end{pg}

\begin{proposition}\label{pushforward functors and Frobenius algebra objects}
	Let $F: \calC^\otimes \rightarrow \calD^\otimes$ be a symmetric monoidal functor between symmetric monoidal stable $\infty$-categories $\calC^\otimes$ and $\calD^\otimes$. Suppose that the underlying $\infty$-categories $\calC$ and $\calD$ are presentable and locally rigid, and that the underlying functor $f: \calC \rightarrow \calD$ preserves small colimits. Then the functor $f$ admits a right adjoint $g$ which exhibits $\calD$ as a Frobenius algebra object of the $\infty$-category $\Mod_{\calC}(\Pres^{\St})$ of $\calC$-module objects of $\Pres^{\St}$.
\end{proposition}

\begin{proof} 
	Since the $\infty$-categories $\calC$ and $\calD$ are presentable and the functor $f$ preserves small colimits, the existence of a right adjoint $g$ follows from the adjoint functor theorem of \cite[5.5.2.9]{MR2522659}. Let us regard $f: \calC \rightarrow \calD$ as a morphism between commutative algebra objects of $\Pres^{\St}$, so that $\calD$ can be viewed as a commutative algebra object of $\Mod_{\calC}(\Pres^{\St})$. We wish to show that the canonical map
$$
e: \calD \otimes_{\calC}\calD \stackrel{\otimes}{\longrightarrow} \calD \stackrel{g}{\longrightarrow} \calC
$$
is a duality datum in $\Mod_{\calC}(\Pres^{\St})$. Using \cite[D.7.7.4]{lurie2018sag} and our assumption that $\calC$ is locally rigid and stable and that $\calD$ is compactly generated (see also \cite[21.1.2.3]{lurie2018sag}), we deduce that $\calD$ admits a dual $\calD^\vee$ in $\Mod_{\calC}(\Pres^{\St})$. Let $\overline{e}: \calD^\vee \otimes_{\calC} \calD \rightarrow \calC$ be a duality datum. Then the universal property of $\overline{e}$ guarantees that there exists a morphism $u:\calD \rightarrow \calD^\vee$ in $\Mod_{\calC}(\Pres^{\St})$ for which the canonical map $e$ factors as a composition
$$
\calD \otimes_{\calC}\calD \stackrel{u\otimes \id}{\longrightarrow} \calD^\vee \otimes_{\calC} \calD \stackrel{\overline{e}}{\longrightarrow} \calC.
$$
It will therefore suffice to show that $u$ is an equivalence. Let $\lambda_{\calC}: \calC \rightarrow \Sp$ and $\lambda_{\calD}: \calD \rightarrow \Sp$ denote the functors given by $\lambda_{\calC}(C)=\underline{\Map}_{\calC}({\bf 1_{\calC}}, C)$ and $\lambda_{\calD}(D)=\underline{\Map}_{\calD}({\bf 1_{\calD}}, D)$, respectively; here we regard the stable $\infty$-category $\calC$ as enriched over the symmetric monoidal $\infty$-category $\Sp$ of \cite[4.8.2.19]{lurie2017ha}, so that for each object $C\in \calC$, there exists a mapping object $\underline{\Map}_{\calC}({\bf 1_{\calC}}, C) \in \Sp$, which is characterized by the following universal property: for every object $X \in \Sp$, we have a canonical homotopy equivalence $\Map_{\Sp}(X, \underline{\Map}_{\calC}({\bf 1_{\calC}}, C)) \simeq \Map_{\calC}(X\otimes {\bf 1_{\calC}}, C)$ (and similarly for $\calD$). Using \cite[D.7.5.1]{lurie2018sag} and invoking again our assumption that $\calC$ is locally rigid, we see that $\lambda_{\calC}: \calC \rightarrow \Sp$ is a nondegenerate map in the sense of \cite[4.6.5.1]{lurie2017ha} (that is, the composition of $\lambda_{\calC}$ with the multiplication map $\calC \otimes \calC \rightarrow \calC$ is a duality datum in $\Pres^{\St}$). It then follows from \cite[4.6.5.14]{lurie2017ha} that the composite map
$$
\calD^\vee \otimes \calD \longrightarrow \calD^\vee \otimes_{\calC} \calD \stackrel{\overline{e}}{\longrightarrow} \calC \stackrel{\lambda_{\calC}}{\longrightarrow} \Sp
$$
is a duality datum in $\Pres^{\St}$. Note that the canonical map $\lambda_{\calD} \rightarrow \lambda_{\calC} \circ g$ is an equivalence of functors, so that the composition of this duality datum with the map $u\otimes \id: \calD \otimes \calD \rightarrow \calD^\vee \otimes \calD$ can be identified with the map
$$
\calD \otimes \calD \stackrel{\otimes}{\longrightarrow} \calD \stackrel{\lambda_{\calD}}{\longrightarrow} \Sp.
$$
This map is a duality datum in $\Pres^{\St}$ by virtue of \cite[D.7.5.1]{lurie2018sag} and our assumption that $\calD$ is locally rigid, so that $u$ is an equivalence as desired. 
\end{proof}

\begin{example}\label{examples for pushforward functors and Frobenius algebra objects}
	The hypotheses of \ref{pushforward functors and Frobenius algebra objects} is satisfied in the situations of interest to us:
\begin{enumerate}[(i)]
\item The pullback functor $f^\ast: \QCoh(Y) \rightarrow \QCoh(X)$ associated to a morphism $f:X \rightarrow Y$ of perfect stacks. This follows from \cite[6.2.3.4]{lurie2018sag}, \cite[\textsection 6.2.6]{lurie2018sag}, \cite[9.1.3.2]{lurie2018sag}, and \cite[9.4.4.5]{lurie2018sag}. Note that every compact object of $\QCoh(X)$ is dualizable by virtue of \cite[9.1.5.2]{lurie2018sag} and \cite[6.2.6.2]{lurie2018sag} (and similarly for $\QCoh(Y)$).

\item The pullback functor $f^\ast: \QCoh(\sfY) \rightarrow \QCoh(\sfX)$ associated to a morphism $f:\sfX \rightarrow \sfY$ of nonconnective spectral Deligne-Mumford stacks (see \cite[1.4.4.2]{lurie2018sag}) for which the $\infty$-categories $\QCoh(\sfX)$ and $\QCoh(\sfY)$ are locally rigid (see \cite[D.7.4.1]{lurie2018sag}). This is a consequence of \cite[2.2.4.1]{lurie2018sag}, \cite[2.2.4.2]{lurie2018sag} and \cite[2.5.0.2]{lurie2018sag}.
\end{enumerate}
\end{example}

\begin{corollary}\label{restriction of a duality datum} 
	Let $F: \calC^\otimes \rightarrow \calD^\otimes$ be as in \emph{\ref{pushforward functors and Frobenius algebra objects}}. Suppose we are given a duality datum $e:\mathcal{M} \otimes_{\calD} \mathcal{N} \rightarrow \calD$ in the symmetric monoidal $\infty$-category $\Mod_{\calD}(\Pres^{\St})$. Then the composite functor
$$
\mathcal{M} \otimes_{\calC} \mathcal{N} \longrightarrow \mathcal{M} \otimes_{\calD} \mathcal{N} \stackrel{e}{\longrightarrow} \calD \stackrel{g}{\longrightarrow} \calC
$$
exhibits $\mathcal{M}$ as a dual of $\mathcal{N}$ in $\Mod_{\calC}(\Pres^{\St})$.
\end{corollary}

\begin{proof}
	By virtue of \ref{pushforward functors and Frobenius algebra objects}, the right adjoint $g: \calD \rightarrow \calC$ is nondegenerate (see \cite[4.6.5.1]{lurie2017ha}) as a morphism in $\Mod_{\calC}(\Pres^{\St})$, so that the desired result follows from \cite[4.6.5.14]{lurie2017ha}. \end{proof}

\begin{pg}
	We are now ready to give the definition of a twisted equivalence: 
\end{pg}

\begin{definition}\label{twisted equivalences}
	Let $f:X \rightarrow Y$ be a morphism of quasi-geometric stacks satisfying condition $(\ast)$ of \ref{promotion of the global sections functors} for which the pullback functor $f^\ast: \QCoh(Y) \rightarrow \QCoh(X)$ satisfies the hypotheses of \ref{pushforward functors and Frobenius algebra objects}. Let $\calC$ and $\calD$ be stable quasi-coherent stacks on $X$ and $Y$ such that the $\infty$-categories of global sections $\QCoh(X; \calC)$ and $\QCoh(Y; \calD)$ are dualizable as modules over $\QCoh(X)$ and $\QCoh(Y)$ in $\Pres^{\St}$, respectively. We let $\ldual{\QCoh(X; \calC)}$ and $\ldual{\QCoh(Y; \calD)}$ denote the duals of $\QCoh(X; \calC)$ and $\QCoh(Y; \calD)$, respectively. Let $F: \QCoh(Y; \calD) \rightarrow \QCoh(X; \calC)$ be an equivalence of $\QCoh(Y)$-module objects of $\Pres^{\St}$ (where we regard $\QCoh(X; \calC)$ as a module over $\QCoh(Y)$ via the pullback functor $f^\ast$), so that \ref{restriction of a duality datum} supplies an equivalence
$$
\ldual{\QCoh(Y; \calD)} \otimes_{\QCoh(Y)} \QCoh(Y; \calD) \rightarrow \ldual{\QCoh(X; \calC)} \otimes_{\QCoh(Y)} \QCoh(X; \calC)
$$  
in $\Mod_{\QCoh(Y)}(\Pres^{\St})$; here we regard $\ldual{\QCoh(X; \calC)}$ as a $\QCoh(Y)$-module via $f^\ast$. We will say that $F$ is a \emph{twisted equivalence} if the composition of this equivalence with the canonical map of relative tensor products
$$
\ldual{\QCoh(X; \calC)} \otimes_{\QCoh(Y)} \QCoh(X; \calC) \rightarrow \ldual{\QCoh(X; \calC)} \otimes_{\QCoh(X)} \QCoh(X; \calC)
$$
formed in $\Pres^{\St}$ is an equivalence of presentable stable $\infty$-categories, which fits into a commutative diagram
$$
\Pull{\QCoh(Y)}{\ldual{\QCoh(Y; \calD)} \otimes_{\QCoh(Y)} \QCoh(Y; \calD)}{\QCoh(X)}{\ldual{\QCoh(X; \calC)} \otimes_{\QCoh(X)} \QCoh(X; \calC),}{}{f^\ast}{}{}
$$
where the horizontal maps are given by duality data in $\Mod_{\QCoh(Y)}(\Pres^{\St})$ and $\Mod_{\QCoh(X)}(\Pres^{\St})$, respectively. 
\end{definition}

\begin{remark}
	In the situation of \ref{twisted equivalences}, the condition that $F: \QCoh(Y; \calD) \rightarrow \QCoh(X; \calC)$ is a twisted equivalence does not depend on the choice of duality data appearing in the above diagram; see \cite[4.6.1.9]{lurie2017ha}. 
\end{remark}

\begin{pg}
	According to \cite[10.2.0.1]{lurie2018sag}, if $\sfX$ is a quasi-geometric spectral Deligne-Mumford stack, then the global sections functor $\QCoh(\sfX; \bullet): \QStk^{\PSt}(\sfX) \rightarrow \Mod_{\QCoh(\sfX)^{\cn}}(\Groth_\infty)$ is fully faithful (here $\QCoh(\sfX)^{\cn}$ denotes the $\infty$-category of connective quasi-coherent sheaves on $\sfX$ of \cite[2.2.5.3]{lurie2018sag}). We will need the following analogue for quasi-geometric stacks in the stable setting, which can be regarded as the special case of \ref{twisted equivalences induce equivalences} where $X$ is a perfect spectral algebraic stack and $f$ is the identity on $X$. Nevertheless, we give a proof of \ref{fully faithfulness of the global sections functor} first, since it will be needed in the proof of \ref{twisted equivalences induce equivalences}.
\end{pg}

\begin{theorem}\label{fully faithfulness of the global sections functor} 
	Let $X$ be a quasi-geometric stack which satisfies condition $(\ast)$ of \emph{\ref{promotion of the global sections functors}}. Then the global sections functor
$$
\QCoh(X; \bullet): \QStk^{\St}(X) \rightarrow \Mod_{\QCoh(X)}(\Pres^{\St})
$$
appearing in \emph{\ref{promotion of the global sections functors}} is fully faithful.
\end{theorem}

\begin{proof} 
	We proceed as in the proof of \cite[10.2.0.1]{lurie2018sag}. Let $\calD \in \Mod_{\QCoh(X)}(\Pres^{\St})$ be an object. For each connective $\bbE_\infty$-ring $A$ and each point $\eta \in X(A)$, let $\calD_\eta$ denote the relative tensor product $\Mod_A \otimes_{\QCoh(X)} \calD$ formed in the $\infty$-category $\Pres^{\St}$. Then the construction $\calD \mapsto \{\calD_\eta\}$ determines a functor $\Psi_X: \Mod_{\QCoh(X)}(\Pres^{\St}) \rightarrow \QStk^{\St}(X)$. Using the equivalence $\Psi_X(\calD) \simeq \calQ_X \otimes_{q^\ast q_\ast \calQ_X} q^\ast \calD$ (where $\calQ_X$ is the unit object of $\QStk^{\St}(X)$), we see that the functor $\Psi_X$ is left adjoint to the global sections functor $\QCoh(X; \bullet)$. It will therefore suffice to show that for every stable quasi-coherent stack $\calC$ on $X$, the counit map $v: \Psi_X(\QCoh(X; \calC)) \rightarrow \calC$ is an equivalence of stable quasi-coherent stacks on $X$. We will prove this by showing that for each $A$-valued point $\eta: \Spec A \rightarrow X$, the pullback $\eta^\ast(v): \Mod_A \otimes_{\QCoh(X)}\QCoh(X; \calC) \rightarrow \calC_\eta$ is an equivalence in $\LinCat^{\St}_A$ of \cite[D.1.5.1]{lurie2018sag}. We note that the underlying map of presentable stable $\infty$-categories can be identified with the canonical map $\Mod_A \otimes_{\QCoh(X)}\QCoh(X; \calC) \rightarrow \QCoh(\Spec A; \calC_\eta)$. The desired result now follows from \cite[5.27]{chough_brauer}, which guarantees that this map is an equivalence (note that the map of quasi-geometric stacks $\eta: \Spec A \rightarrow X$ is quasi-affine because $X$ is quasi-geometric). 
\end{proof}

\begin{pg}
	Before giving the proof of \ref{twisted equivalences induce equivalences}, we need to recall a bit of terminology. Let $R$ be a connective $\bbE_\infty$-ring and let $A$ be an $\bbE_1$-algebra over $R$ (see \cite[7.1.3.5]{lurie2017ha}). Recall from \cite[11.5.3.1]{lurie2018sag} that $A$ is said to be an \emph{Azumaya algebra over R} if the underlying $R$-module of $A$ is a compact generator of $\Mod_R$ and the map $A\otimes_R A^{\rev} \rightarrow \End_R(A)$ induced by the left and right actions of $A$ on itself is an equivalence, where $A^{\rev}$ denotes the opposite algebra of $A$ (see \cite[4.1.1.7]{lurie2017ha}). According to \cite[11.5.3.4]{lurie2018sag}, $A$ is an Azumaya algebra over $R$ if and only if $\RMod_A$ is invertible as an object of $\LinCat^{\St}_R$. 
\end{pg}

\begin{remark}\label{Azumaya algebras and invertibility in the nonconnective case}
	The definition of \cite[11.5.3.1]{lurie2018sag} and the proof of \cite[11.5.3.4]{lurie2018sag} are valid also when $R$ is not connective (see also \cite[3.13]{MR3190610} and \cite[3.15]{MR3190610}), so that the hypothesis that $R$ is connective is not necessary in the statement of \cite[11.5.3.4]{lurie2018sag}. 
\end{remark}

\begin{pg}
	Let $X: \CAlg^{\cn} \rightarrow \widehat{\SSet}$ be a functor. Let $\QStk^{\cg}(X)$ denote the $\infty$-category of compactly generated stable quasi-coherent stacks on $X$, which we regard as a symmetric monoidal $\infty$-category; see \cite[10.3.1.3]{lurie2018sag} and \cite[11.4.0.1]{lurie2018sag}. Recall from \cite[11.5.2.1]{lurie2018sag} that the \emph{extended Brauer space $\sBr^\dagger(X)$} of $X$ is defined to be the full subcategory of $\QStk^{\cg}(X)^\simeq$ spanned by the invertible objects of the symmetric monoidal $\infty$-category $\QStk^{\cg}(X)$, where $\QStk^{\cg}(X)^\simeq$ is the largest Kan complex contained in $\QStk^{\cg}(X)$. We denote the set of connected components of $\sBr^\dagger(X)$ by $\Br^\dagger(X)$ and refer to it as the \emph{extended Brauer group} of $X$. According to \cite[11.5.3.7]{lurie2018sag}, if $\calA$ is an $\bbE_1$-algebra object of $\QCoh(X)$, then $\calA$ is said to be an \emph{Azumaya algebra on $X$} if, for each connective $\bbE_\infty$-ring $R$ and each point $\eta \in X(R)$ (which we identify with a morphism $\Spec R \rightarrow X$), the pullback $\eta^\ast \calA$ is an Azumaya algebra over $R$ in the sense of \cite[11.5.3.1]{lurie2018sag}. Invoking \cite[11.5.3.9]{lurie2018sag}, we can associate to each Azumaya algebra $\calA$ on $X$ an object of $\sBr^\dagger(X)$, given by the formula $(\eta: \Spec R \rightarrow X) \mapsto (\RMod_{\eta^\ast \calA} \in \LinCat^{\St}_R)$. We denote the connected component of this object by $[\calA]$ and refer to it as the \emph{extended Brauer class} of $\calA$.
\end{pg}	

\begin{pg}
	To prove \ref{twisted equivalences induce equivalences}, we will need a global version of \cite[11.5.3.4]{lurie2018sag}:
\end{pg}

\begin{lemma}\label{Azumaya algebras and invertibility in the global case} 
	Let $X$ be a quasi-geometric stack satisfying condition $(\ast)$ of \emph{\ref{promotion of the global sections functors}}, and let $\calA$ be an $\bbE_1$-algebra object of $\QCoh(X)$. Suppose that $X$ is perfect in the sense of \emph{\cite[9.4.4.1]{lurie2018sag}}. Then $\calA$ is an Azumaya algebra on $X$ if and only if the $\infty$-category $\RMod_{\calA}(\QCoh(X))$ of right $\calA$-module objects of $\QCoh(X)$ is invertible as an object of $\Mod_{\QCoh(X)}(\Pres^{\St})$. 
\end{lemma}

\begin{proof}
	Since $\QCoh(X)$ is a presentable stable $\infty$-category (see \cite[9.1.3.2]{lurie2018sag} and \cite[6.2.3.4]{lurie2018sag}), so is $\RMod_{\calA}(\QCoh(X))$ by virtue of \cite[4.2.3.7]{lurie2017ha} and \cite[7.1.1.4]{lurie2017ha}. Let $\LMod_{\calA}(\QCoh(X))$ denote the $\infty$-category of left $\calA$-module objects of $\QCoh(X)$. According to \cite[4.8.4.6]{lurie2017ha}, the tensor product $\LMod_{\calA}(\QCoh(X)) \otimes_{\QCoh(X)}\RMod_{\calA}(\QCoh(X))$ can be identified with the $\infty$-category $\BMod{\calA}{\calA}(\QCoh(X))$ of $\calA$-$\calA$-bimodule objects of $\QCoh(X)$; see \cite[4.3.1.12]{lurie2017ha}. It then follows from \cite[4.8.4.8]{lurie2017ha} that the unit object $\calA \in \BMod{\calA}{\calA}(\QCoh(X))$ determines a $\QCoh(X)$-linear functor 
$$
c: \QCoh(X) \rightarrow \LMod_{\calA}(\QCoh(X)) \otimes_{\QCoh(X)}\RMod_{\calA}(\QCoh(X)),
$$
which exhibits $\LMod_{\calA}(\QCoh(X))$ as a $\QCoh(X)$-linear dual of $\RMod_{\calA}(\QCoh(X))$, so that $\RMod_{\calA}(\QCoh(X))$ is invertible in $\Mod_{\QCoh(X)}(\Pres^{\St})$ if and only if the functor $c$ is an equivalence. Applying \cite[2.1.3]{lurie2017br}, we deduce that the latter condition holds if and only if the following conditions are satisfied: 
\begin{enumerate}[(i)]
\item The sheaf $\calA$ is dualizable as an object of $\QCoh(X)$.
\item The map $\calA \otimes \calA^{\rev} \rightarrow \End(\calA)$ induced by the left and right actions of $\calA$ on itself is an equivalence (here $\calA^{\rev}$ denotes the opposite algebra of $\calA$; see \cite[4.1.1.7]{lurie2017ha}). 
\item The functor $\QCoh(X) \rightarrow \QCoh(X)$ given on objects by the formula $\calF \mapsto \calA \otimes \calF$ is conservative. 
\end{enumerate}
Using \cite[6.2.6.2]{lurie2018sag}, we see that the first condition holds if and only if, for every connective $\bbE_\infty$-ring $R$ and every $R$-valued point $\eta \in X(R)$, the $R$-module spectrum $\calA_\eta$ is dualizable. Note that since $X$ is a perfect stack which satisfies condition $(\ast)$ of \ref{promotion of the global sections functors}, for every $R$-point $\eta: \Spec R \rightarrow X$, the canonical map $\eta^\ast \End(\calA) \rightarrow \End(\eta^\ast \calA)$ is an equivalence of $\bbE_1$-algebras over $R$ provided that $\calA$ is compact (or equivalently, dualizable; see \cite[9.1.5.3]{lurie2018sag} and \cite[6.2.6.2]{lurie2018sag}). Using these observations and applying \cite[2.1.3]{lurie2017br} again, we can reduce to the case where $X$ is affine, in which case the desired result follows from \cite[11.5.3.4]{lurie2018sag}. 
\end{proof}

\begin{pg}
	We are now ready to prove our main result \ref{twisted equivalences induce equivalences}:
\end{pg}

\begin{proof}[Proof of \emph{\ref{twisted equivalences induce equivalences}}]
	We first note that $X$ and $Y$ satisfy condition $(\ast)$ of \ref{promotion of the global sections functors} and that the pullback functor $f^\ast: \QCoh(Y) \rightarrow \QCoh(X)$ satisfies the hypotheses of \ref{pushforward functors and Frobenius algebra objects} (see \ref{examples for pushforward functors and Frobenius algebra objects} and \ref{perfect spectral algebraic stacks over the sphere spectrum}), so that it makes sense to consider the twisted equivalence $F$. Since $X$ and $Y$ are quasi-geometric spectral algebraic stacks which admit a quasi-finite presentation (in the sense of \cite[1.3]{chough_brauer}), we deduce from \cite[1.2]{chough_brauer} that there exist Azumaya algebras $\calA$ and $\calB$ on $X$ and $Y$ (in the sense of \cite[11.5.3.7]{lurie2018sag}) for which the extended Brauer classes $[\calA]$ and $[\calB]$ (in the sense of \cite[11.5.3.9]{lurie2018sag}) are equal to the connected components of $\calC$ and $\calD$ in the extended Brauer groups $\Br^\dagger(X)$ and $\Br^\dagger(Y)$ of \cite[11.5.2.1]{lurie2018sag}, respectively. Consequently, we can identify $\QCoh(X; \calC)$ and $\QCoh(Y; \calD)$ with $\RMod_{\calA}(\QCoh(X))$ and $\RMod_{\calB}(\QCoh(Y))$ in $\Mod_{\QCoh(X)}(\Pres^{\St})$ and $\Mod_{\QCoh(Y)}(\Pres^{\St})$, respectively. Our assumption that $F$ is a twisted equivalence supplies a commutative diagram
$$
\Pull{\QCoh(Y)}{\LMod_{\calB}(\QCoh(Y)) \otimes_{\QCoh(Y)} \RMod_{\calB}(\QCoh(Y))}{\QCoh(X)}{\LMod_{\calA}(\QCoh(X)) \otimes_{\QCoh(X)} \RMod_{\calA}(\QCoh(X)),}{}{f^\ast}{}{}
$$
where the horizontal maps exhibit $\LMod_{\calB}(\QCoh(Y))$ and $\LMod_{\calA}(\QCoh(X))$ as duals of $\RMod_{\calB}(\QCoh(Y))$ and $\RMod_{\calA}(\QCoh(X))$ in $\Mod_{\QCoh(Y)}(\Pres^{\St})$ and $\Mod_{\QCoh(X)}(\Pres^{\St})$, respectively. Using \ref{Azumaya algebras and invertibility in the global case}, we see that the horizontal maps are equivalences. Combining this with our assumption that $F$ is a twisted equivalence (which guarantees that the right vertical map is an equivalence), we deduce that the left vertical map $f^\ast: \QCoh(Y) \rightarrow \QCoh(X)$ is also an equivalence. Then \cite[9.2.2.1]{lurie2018sag} guarantees that the map $f$ is an equivalence, thereby completing the proof of the first assertion. To prove the second, we note that since $f$ is an equivalence, the twisted equivalence $F$ induces an equivalence $\QCoh(X; f^\ast \calD) \simeq \QCoh(X; \calC)$ of $\QCoh(X)$-module objects of $\Pres^{\St}$, so that we obtain an equivalence $f^\ast \calD \simeq \calC$ of stable quasi-coherent stacks on $X$ by virtue of \ref{fully faithfulness of the global sections functor}. 
\end{proof}

\section{Twisted Equivalences for $\bbE_\infty$-rings}

	The goal of this section is to prove \ref{twisted equivalences in the case of nonconnective affines}, an analogue of \cite[9.2]{MR1002456} in the setting of $\bbE_\infty$-rings. 

\begin{pg}
	If $A$ is an ordinary associative ring, we let $\DD^b(A)$ denote the bounded derived category of $A$-modules. Recall that the \emph{center} of $A$, which we denote by $\mathrm{Z}(A)$, is defined as the subalgebra consisting of those elements $x \in A$ for which $xy=yx$ for every $y \in A$. Let $A$ and $A'$ be ordinary associative rings and suppose we are given an equivalence of triangulated categories $\DD^b(A) \simeq \DD^b(A')$. It then follows from \cite[9.2]{MR1002456} of Jeremy Rickard that there is an isomorphism of algebras $\mathrm{Z}(A) \simeq \mathrm{Z}(A')$. To establish a spectral analogue of this result, we first need to adapt the definition of centers to the setting of $\bbE_\infty$-rings. To this end, let $R$ be an $\bbE_\infty$-ring. According to \cite[5.3.1.12]{lurie2017ha}, we can associate to each $\bbE_1$-algebra $A$ over $R$ an $\bbE_2$-algebra over $R$, which we denote by $\Z(A)$ and refer to as the \emph{center} of $A$. More informally, $\Z(A)$ is an $\bbE_2$-algebra over $R$ which is universal among those $\bbE_2$-algebras $B$ over $R$ for which $A$ can be promoted to an associative algebra object of $\LMod_B$; see \cite[5.3.1.14]{lurie2017ha}. Let $\BMod{A}{A}(\Mod_R)$ denote the $\infty$-category of $A$-$A$-bimodule objects of $\Mod_R$, which we regard as an $\infty$-category tensored over $\Mod_R$; see \cite[4.3.1.12]{lurie2017ha}. It then follows from \cite[5.3.1.30]{lurie2017ha} and \cite[4.4.1.28]{lurie2017ha} that the underlying $\bbE_1$-algebra of $\Z(A)$ over $R$ can be identified with a classifying object for endomorphisms of $A$ as an object of $\BMod{A}{A}(\Mod_R)$. 
\end{pg}

\begin{remark}
	The notation $\Z(A)$ is somewhat abusive, since it depends on $R$ in the following sense: if $R' \rightarrow R$ is a morphism of $\bbE_\infty$-rings, then it is not necessarily true that the induced map $\Z(\underline{A}) \rightarrow \underline{\Z(A)}$ is an equivalence of $\bbE_2$-algebras over $R'$, where $\underline{A}$ and $\underline{\Z(A)}$ denote the underlying $\bbE_1$-algebra of $A$ over $R'$ and the underlying $\bbE_2$-algebra of $\Z(A)$ over $R'$, respectively (see \ref{the centers and the forgetful functors}).  
\end{remark}

\begin{remark}
	Let $A$ be an ordinary associative ring, which we regard as an ordinary associative algebra over the ordinary commutating ring of integers $\bbZ$. Let us regard $\bbZ$ as a discrete $\bbE_\infty$-ring, so that we can regard $A$ as a discrete $\bbE_1$-algebra over $\bbZ$ (see \cite[7.1.3.18]{lurie2017ha}). Then there are two different ways to define the notion of a center of $A$: the ordinary center $\mathrm{Z}(A)$ (where we regard $A$ as an ordinary associative ring) and the center $\Z(A)$ (where we regard $A$ as a discrete $\bbE_1$-algebra over $\bbZ$). The latter is generally not discrete (unlike the former) and has a rich structure: for example, we can identify the homotopy groups $\pi_n \Z(A)$ with the \emph{$n$th Hochschild cohomology groups} of $A$. In particular, we have an isomorphism $\pi_0\Z(A) \simeq \mathrm{Z}(A)$.
\end{remark}

\begin{pg}\label{setup for twisted equivalences for nonconnective affines} 
	In the formulation of \ref{twisted equivalences in the case of nonconnective affines}, we make use of the observation that the definition of twisted equivalences of \ref{twisted equivalences} can be extended to the nonconnective setting. In the special case of affine nonconnective spectral Deligne-Mumford stacks (see \cite[\textsection 1.4.7]{lurie2018sag}), it can be described explicitly as follows: suppose we are given a map of (not necessarily connective) $\bbE_\infty$-rings $f: R' \rightarrow R$. Combining \cite[7.2.4.2]{lurie2017ha}, \cite[4.4.2.15]{lurie2017ha}, and \cite[7.2.4.4]{lurie2017ha}, we see that the symmetric monoidal stable $\infty$-categories $\Mod_{R'}$ and $\Mod_R$ are locally rigid (see \cite[D.7.4.1]{lurie2018sag}). Using \ref{examples for pushforward functors and Frobenius algebra objects}, we deduce that the symmetric monoidal functor $f^\ast: \Mod_{R'} \rightarrow \Mod_R$ (which carries an $R'$-module spectrum $M'$ to the relative tensor product $R\otimes_{R'}M'$) satisfies the hypotheses of \ref{pushforward functors and Frobenius algebra objects}. Now let $A'$ and $A$ be $\bbE_1$-algebras over $R'$ and $R$, respectively, and let $F:\RMod_{A'} \rightarrow \RMod_A$ be an equivalence of stable $R'$-linear $\infty$-categories (where we regard $\RMod_A$ as a stable $R'$-linear $\infty$-category via the extension of scalars functor $f^\ast$). Note that \cite[4.8.4.6]{lurie2017ha} supplies an equivalence of $R'$-linear $\infty$-categories $\LMod_{A'} \otimes_{R'} \RMod_{A'} \simeq \BMod{A'}{A'}(\Mod_{R'})$, where $\BMod{A'}{A'}(\Mod_{R'})$ denotes the $\infty$-category of $A'$-$A'$-bimodule objects of $\Mod_{R'}$. Using this equivalence and \cite[4.8.4.8]{lurie2017ha}, we see that the $R'$-linear functor $c_{A'}: \Mod_{R'} \rightarrow \LMod_{A'} \otimes_{R'} \RMod_{A'}$ determined by the unit object $A' \in \BMod{A'}{A'}(\Mod_{R'})$ is a duality datum in the symmetric monoidal $\infty$-category $\LinCat^{\St}_{R'}$ of \cite[D.1.5.1]{lurie2018sag}. Similarly, the $R$-linear functor $c_A: \Mod_R \rightarrow \LMod_A\otimes_R \RMod_A$ determined by $A \in \BMod{A}{A}(\Mod_R)$ is a duality datum in $\LinCat^{\St}_R$. Moreover, \ref{restriction of a duality datum} guarantees that $\RMod_A \in \LinCat^{\St}_R$ is also dualizable when regarded as a stable $R'$-linear $\infty$-category via the functor $f^\ast$, and its dual can be identified with $\LMod_A$ viewed as an object of $\LinCat^{\St}_{R'}$ via the functor $f^\ast$. Consequently, the functor $F$ induces an equivalence of $\infty$-categories $\LMod_{A'} \otimes_{R'} \RMod_{A'} \rightarrow \LMod_A \otimes_{R'} \RMod_A$. Composing with the canonical map $\LMod_A \otimes_{R'} \RMod_A \rightarrow \LMod_A \otimes_R \RMod_A$, we obtain a functor
$$
\LMod_{A'} \otimes_{R'} \RMod_{A'} \rightarrow \LMod_A \otimes_R \RMod_A. 
$$
We will say that the functor $F:\RMod_{A'} \rightarrow \RMod_A$ is a \emph{twisted equivalence} if this composite functor is an equivalence of $\infty$-categories fitting into a commutative diagram
$$
\Pull{\Mod_{R'}}{\LMod_{A'} \otimes_{R'} \RMod_{A'}}{\Mod_R}{\LMod_A \otimes_R \RMod_A.}{c_{A'}}{f^\ast}{}{c_A}
$$
\end{pg}

\begin{pg}
	Let $k$ be a nonnegative integer. If $R$ is an $\bbE_\infty$-ring, we will denote the $\infty$-category $\Alg_{\bbE_k}(\Mod_R)$ of $\bbE_k$-algebra objects of $\Mod_R$ by $\Alg^{(k)}_R$; see \cite[7.1.3.5]{lurie2017ha}. If $k=1$, we will denote $\Alg^{(k)}_R$ simply by $\Alg_R$. 

	We are now ready to give the proof of \ref{twisted equivalences in the case of nonconnective affines}:
\end{pg}

\begin{proof}[Proof of \emph{\ref{twisted equivalences in the case of nonconnective affines}}]
	Using \cite[5.3.1.30]{lurie2017ha} and \cite[4.4.1.28]{lurie2017ha}, we see that the underlying $\bbE_1$-algebra of $\Z(A)$ over $R$ can be identified with a classifying object for endomorphisms of $A$ as an object of the $R$-linear $\infty$-category $\BMod{A}{A}(\Mod_R)$, which we denote by $\End_{\BMod{A}{A}(\Mod_R)}(A)$ (and similarly for $\Z(A')$). Since the forgetful functor $\Alg^{(2)}_{R'} \rightarrow \Alg_{R'}$ is conservative (see \cite[3.2.2.6]{lurie2017ha} and \cite[5.1.2.2]{lurie2017ha}), it will suffice to show that the induced map $\End_{\BMod{A'}{A'}(\Mod_{R'})}(A') \rightarrow \End_{\BMod{A}{A}(\Mod_R)}(A)$ is an equivalence of $\bbE_1$-algebras over $R'$, where we regard the latter classifying object as an $\bbE_1$-algebra over $R'$ via the map $f: R' \rightarrow R$. This follows immediately from the commutativity of the diagram appearing in \ref{setup for twisted equivalences for nonconnective affines}. 
\end{proof}

\begin{corollary}\label{twisted equivalences for Azumaya algebras in the case of nonconnective affines}
	In the situation of \emph{\ref{twisted equivalences in the case of nonconnective affines}}, suppose that $A'$ and $A$ are Azumaya algebras over $R'$ and $R$, respectively. Then $f$ is an equivalence.
\end{corollary}

\begin{proof}
	Since $A \in \Alg_R$ is an Azumaya algebra over $R$, it follows from \ref{Azumaya algebras and invertibility in the nonconnective case} that $\RMod_A$ is invertible as an object of $\LinCat^{\St}_R$, so that the lower horizontal map $c_A$ in the diagram appearing in \ref{setup for twisted equivalences for nonconnective affines} is an equivalence of $\infty$-categories. Consequently, the unit map $R \rightarrow \Z(A)$ is an equivalence of $\bbE_2$-algebras over $R$. Similarly, we can identify $\Z(A')$ with $R'$ via the unit map $R' \rightarrow \Z(A')$. The desired result now follows from \ref{twisted equivalences in the case of nonconnective affines}. 
\end{proof}

\bibliography{chough_twisted}
\bibliographystyle{amsplain}

\end{document}